\documentclass[12pt,twoside,reqno]{amsart}
\usepackage{amsmath}
\usepackage{amsfonts}
\usepackage{amssymb}
\usepackage{color}
\usepackage{mathrsfs}
\usepackage{hyperref}
\usepackage{cite}
\usepackage{geometry}
\usepackage{marginnote}%\marginpar{}, \marginnote{}
\usepackage{todonotes}%\todo{}
\allowdisplaybreaks
\newtheorem{theorem}{Theorem}[section]

\newtheorem{lemma}[theorem]{Lemma}
\newtheorem{proposition}[theorem]{Proposition}
\newtheorem{definition}[theorem]{Definition}
\newtheorem{remark}[theorem]{Remark}
\numberwithin{equation}{section}
\begin{document}
\title{Stationary solutions to Smoluchowski's coagulation equation with source} 
%\thanks{}
\author{Philippe Lauren\c{c}ot}
\address{Institut de Math\'ematiques de Toulouse, UMR~5219, Universit\'e de Toulouse, CNRS \\ F--31062 Toulouse Cedex 9, France}
\email{laurenco@math.univ-toulouse.fr}

%\author{}
%\address{}
%\email{}


\date{\today}

%%%%%%%%%%%%%%%%
%%%%%%%%%%%%%%%%
\begin{abstract}
Existence and non-existence of integrable stationary solutions to Smoluchowski's coagulation equation with source are investigated when the source term is integrable with an arbitrary support in $(0,\infty)$. Besides algebraic upper and lower bounds, a monotonicity condition is required for the coagulation kernel. Connections between integrability properties of the source and the corresponding stationary solutions are also studied.
\end{abstract}
%%%%%%%%%%%%%%%%
%%%%%%%%%%%%%%%%

\maketitle

%
%     HEADLINES
%
\pagestyle{myheadings}
\markboth{\sc{Ph. Lauren\c cot}}{\sc{Stationary solutions to Smoluchowski's coagulation equation with source}}

%%%%%%%%%%%%%%%%
%%%%%%%%%%%%%%%%
\section{Introduction}\label{sec1}
%%%%%%%%%%%%%%%%
%%%%%%%%%%%%%%%%

The coagulation equation with source describes the dynamics of a system of particles, in which particles interact by pairwise merging, thereby forming larger particles, and new particles are injected from the outside. Denoting the particle size distribution function of particles with size $x\in (0,\infty)$ at time $t>0$ by $f=f(t,x)\ge 0$, the corresponding evolution equation is
\begin{subequations}\label{scews}
\begin{align}
\partial_t f(t,x) & = \mathcal{C}f(t,x) + S(x)\,, \qquad (t,x)\in (0,\infty)^2\,, \label{scews1} \\
f(0,x) & = f^{in}(x)\,, \qquad x\in (0,\infty)\,, \label{scews2}
\end{align}
where $S$ is a time-independent function accounting for the external supply of particles and the coagulation mechanism is given by the nonlinear integral operator
\begin{equation}
\mathcal{C}f(x) := \frac{1}{2} \int_0^x K(y,x-y) f(y) f(x-y)\ \mathrm{d}y - \int_0^\infty K(x,y) f(x) f(y)\ \mathrm{d}y \label{scews3}
\end{equation}
\end{subequations}
for $x\in (0,\infty)$. In \eqref{scews3}, the coagulation kernel $K$ is a non-negative and symmetric function and $K(x,y)=K(y,x)$ measures the rate of merging of particles with respective sizes $x$ and $y$. The first integral on the right hand side of \eqref{scews3} accounts for the formation of particles with size $x$ as a result of the coagulation of two particles with respective sizes $y\in (0,x)$ and $x-y$, while the second one describes the disappearance of particles with size $x$ when merging with other particles.

Since the pioneering works \cite{BaCa1990, LeTs1981a, McLe1962c, McLe1964, Melz1957b, Spou1984, Stew1989, Whit1980}, Smoluchowski's coagulation equation \eqref{scews} without source ($S\equiv 0$), originally derived in \cite{Smol1916, Smol1917}, has been extensively studied in the mathematical literature for various choices of the coagulation kernel $K$ and we refer to the books \cite{BLL2019, Dubo1994b} and the references therein for a more detailed account. Since the addition of a source term does not change the  mathematical structure of the equation, the well-posedness of Smoluchowski's coagulation equation with source \eqref{scews} can be proved in a similar way as that of Smoluchowski's coagulation equation \cite{EsMi2006, KuTh2019, Spou1985b, SvR2002, Whit1980}. It is however worth emphasizing that the presence of a source drastically changes the dynamics,  as the continuous injection of new particles in the system somewhat balances the transfer of matter towards larger and larger particles due to coagulation. In particular, convergence to a stationary state is shown in \cite{Dubo1994b, SvR2002, Simo1998} for the constant coagulation kernel, a feature which leads to the question of existence and stability of stationary solutions for other choices of coagulation kernels. A thorough study of the existence issue is performed in \cite{FLNV} for coagulation kernels satisfying 
\begin{equation}
k_1 \left( x^{\gamma+\alpha} y^{-\alpha} + x^{-\alpha} y^{\gamma+\alpha} \right) \le K(x,y) \le k_2 \left( x^{\gamma+\alpha} y^{-\alpha} + x^{-\alpha} y^{\gamma+\alpha} \right)\,, \qquad (x,y)\in (0,\infty)^2\,, \label{flnv1}
\end{equation} 
where $(\gamma,\alpha)\in \mathbb{R}^2$ and $k_2>k_1>0$. Assuming that the source term $S$ is a non-negative bounded Radon measure on $(0,\infty)$ with compact support in $[1,L]$ for some $L>1$, the existence of at least one non-negative measure-valued stationary solution $f(\mathrm{d}x)$ to \eqref{scews1} satisfying
\begin{equation}
\int_0^\infty \left( x^{\gamma+\alpha} + x^{-\alpha} \right) f(\mathrm{d}x) < \infty \label{flnv2}
\end{equation}
is shown in \cite[Theorem~2.2]{FLNV} when $|\gamma+2\alpha|<1$. In addition,
\begin{equation}
\int_0^\infty x^\mu f(\mathrm{d}x) < \infty\,, \quad \mu<\frac{1+\gamma}{2}\,, \qquad \int_0^\infty x^{(1+\gamma)/2} f(\mathrm{d}x) = \infty\,, \label{flnv3}
\end{equation}
see \cite[Corollary~6.4]{FLNV}, so that $f(\mathrm{d}x)$ cannot decay too fast for large sizes (observe that the condition $|\gamma+2\alpha|<1$ implies that $\max\{\gamma+\alpha,-\alpha\}<(1+\gamma)/2$). Furthermore, if $S\not\equiv 0$ and $|\gamma+2\alpha|\ge 1$, then there is no non-negative measure-valued stationary solution to \eqref{scews1}  satisfying \eqref{flnv2}, see \cite[Theorem~2.4]{FLNV}. 

\bigskip

The purpose of this note is twofold: on the one hand, for coagulation kernels satisfying \eqref{flnv1}, we extend the validity of the existence and non-existence results established in \cite{FLNV} to source terms which are not necessarily compactly supported in $(0,\infty)$. We however restrict the analysis to source terms and stationary solutions which are absolutely continuous with respect to the Lebesgue measure on $(0,\infty)$ and, to this end, an additional monotonicity condition is required on the coagulation kernel. On the other hand, for such source terms, we provide alternative proofs for the existence and non-existence results established in \cite{FLNV}.

\bigskip

We actually begin our analysis with the following observation, already pointed out in \cite[Chapter~8]{Dubo1994b}. If $K$ is a coagulation kernel satisfying \eqref{flnv1} and $f$ is a stationary solution to \eqref{scews1}, then $f_\theta(x):= x^\theta f(x)$, $x>0$, is a stationary solution to \eqref{scews1} with coagulation kernel $K_\theta(x,y) := (xy)^{-\theta} K(x,y)$ and $\theta:= \min\{ \gamma+\alpha,-\alpha \}$, and $K_\theta$ satisfies the growth condition \eqref{flnv1} with $(|\gamma+2\alpha|,0)$ instead of $(\gamma,\alpha)$. Thanks to this observation, we shall assume from now on that there are $\lambda\ge 0$ and $k_2>k_1>0$ such that the coagulation kernel $K$ satisfies
\begin{equation}
k_1 \left( x^\lambda + y^\lambda \right) \le K(x,y) \le k_2 \left( x^\lambda + y^\lambda \right)\,, \qquad (x,y)\in (0,\infty)^2\,. \label{hypK}
\end{equation}
We supplement \eqref{hypK} with the following monotonicity condition on $K$
\begin{equation}
K(x-y,y)\le K(x,y)\,, \qquad 0 < y \le x\,, \label{monK}
\end{equation}
which is known to play an important role in the derivation of uniform integrability estimates such as $L^p$-estimates, $p>1$, see \cite{Buro1983, Dubo1994b, LaMi02c, MiRR03}.

Before providing a precise definition of stationary solutions to \eqref{scews1} along with the statements of the main results, let us introduce some notation: for $m\in\mathbb{R}$, we set $X_m := L^1((0,\infty),x^m \mathrm{d}x)$ and $X_{0,m} := X_0\cap X_m$, and denote their respective positive cones by $X_m^+$ and $X_{0,m}^+$. For $h\in X_m$, we put
\begin{equation*}
M_m(h) = \int_0^\infty x^m h(x)\ dx\,, \qquad h\in X_m\,.
\end{equation*}

\bigskip

We now define the notion of weak stationary solutions to the coagulation equation with source \eqref{scews1} to be used in the sequel. Besides the required absolute continuity with respect to the Lebesgue measure, it is quite similar to \cite[Definition~2.1]{FLNV}.

%%%%%%%%%%%%%%%%
\begin{definition}\label{defD1}
Let $\lambda\ge 0$ and consider a coagulation kernel $K$ satisfying \eqref{hypK} and $S\in X_0^+$. A stationary solution to the coagulation equation with source \eqref{scews1} is a function $\varphi\in X_{0,\lambda}^+$ such that
\begin{equation}
\frac{1}{2} \int_0^\infty \int_0^\infty \chi_\vartheta(x,y) K(x,y) \varphi(x)\varphi(y)\ \mathrm{d}y \mathrm{d}x + \int_0^\infty S(x) \vartheta(x)\ \mathrm{d}x = 0 \label{D1}
\end{equation}
for all $\vartheta\in L^\infty(0,\infty)$, where
\begin{equation}
\chi_\vartheta(x,y) := \vartheta(x+y) - \vartheta(x) - \vartheta(y)\,, \qquad (x,y)\in (0,\infty)^2\,. \label{D1a}
\end{equation}
\end{definition}
%%%%%%%%%%%%%%%%

We now state the existence and non-existence results we establish in this paper.

%%%%%%%%%%%%%%%%
\begin{theorem}\label{thmD0}
Let $\lambda\ge 0$ and consider a coagulation kernel $K$ satisfying \eqref{hypK} and $S\in X_0^+$.
\begin{itemize}
	\item[(a)] Assume further that $K$ satisfies \eqref{monK} and $S\in \bigcap_{m\in (0,1)} X_m$. If $\lambda\in [0,1)$, then there is at least one stationary solution $\varphi$ to \eqref{scews1} such that 
	\begin{equation}
	\varphi\in X_m\,, \qquad 0 \le m < \frac{1+\lambda}{2}\,, \qquad \varphi\not\in X_{(1+\lambda)/2}\,. \label{D100}
	\end{equation}
	In addition, if $S\in X_m$ for some $m\in (-\infty,0)$, then $\varphi\in X_m$.
	\item[(b)] If $\lambda\ge 1$ and $\varphi$ is a stationary solution to \eqref{scews1}, then $\varphi=S\equiv 0$.
\end{itemize}
\end{theorem}
%%%%%%%%%%%%%%%%

An alternative formulation of Theorem~\ref{thmD0}~(b) is that, for $\lambda\ge 1$ and $S\in X_0^+$, $S\not\equiv 0$, there is no stationary solution to \eqref{scews1} in the sense of Definition~\ref{defD1}. 

%%%%%%%%%%%%%%%%
\begin{remark}\label{remX}
According to the above mentioned connection between stationary solutions to \eqref{scews1} for coagulation kernels satisfying \eqref{flnv1} and \eqref{hypK}, existence and non-existence results of stationary solutions to \eqref{scews1} for coagulation kernels satisfying \eqref{flnv1} can be deduced from Theorem~\ref{thmD0}. Indeed, consider a coagulation kernel $K$ satisfying \eqref{flnv1} and $(x-y)^{-\theta} K(x-y,y) \le x^{-\theta} K(x,y)$ for $(x,y)\in (0,\infty)^2$ and $\theta=\min\{\gamma+\alpha,-\alpha\}$. Then, given a source term $S\in \bigcap_{m\in [0,1)} X_m^+$, $S\not\equiv 0$, there is at least one stationary solution to \eqref{scews1} which belongs to $X_m^+$ for $m\in [\theta,(1+\gamma)/2)$, but not to $X_{(1+\gamma)/2}$, when $|\gamma+2\alpha|\in [0,1)$ and no such solution when $|\gamma+2\alpha|\ge 1$. This is in accordance with the results established in \cite{FLNV}.
\end{remark}
%%%%%%%%%%%%%%%%

As already mentioned, the outcome of Theorem~\ref{thmD0} matches the results obtained in \cite{FLNV} for source terms which are non-negative bounded Radon measures on $(0,\infty)$ with compact support in $(0,\infty)$. We here relax the assumption on the support and obtain directly integrable stationary solutions to \eqref{scews1} when the source term is integrable. Also, the proof of Theorem~\ref{thmD0} provided below relies rather on global integral estimates, while local integral estimates are at the forefront of the analysis performed in \cite{FLNV}. As a consequence, more precise information on the local behaviour of stationary solutions is obtained, see \cite[Proposition~6.3]{FLNV}. Finally, as already pointed out in \cite{Dubo1994b, FLNV}, the non-integrability property stated in \eqref{D100} is a striking feature of stationary solutions to \eqref{scews1} as their decay at infinity is prescribed solely by the growth of the coagulation kernel and is not sensitive to the decay at infinity of the source term.

We now describe the contents of this paper. In Section~\ref{sec2}, we derive properties of stationary solutions $\varphi$ to \eqref{scews1} in the sense of Definition~\ref{defD1}, including the non-integrability property $\varphi\not\in X_{(1+\lambda)/2}$ (Proposition~\ref{prc1}) and improved integrability properties of $\varphi$ for small sizes induced by that of the source term (Proposition~\ref{prc2}). We also derive in Proposition~\ref{prD2} upper and lower bounds on $M_0(\varphi)$ and $M_\lambda(\varphi)$ in terms of $M_0(S)$ and $M_\lambda(S)$. Though not directly used in the subsequent analysis, these estimates, in particular \eqref{D2b}, provide guidelines for the proof of Theorem~\ref{thmD0}~(a), see Lemma~\ref{lemb1} and Lemma~\ref{lemb4b}. Section~\ref{sec3} is devoted to the existence of stationary solutions (Theorem~\ref{thmD0}~(a)) and combines a dynamical approach and a compactness method, an approach which has already proved successful to construct self-similar solutions to Smoluchowski's coagulation equation \cite{EsMi2006, EMRR2005, FoLa2005, NiVe2013a} and stationary solutions to coagulation-fragmentation equations \cite{EMRR2005, Laur2019}. Specifically, given a small parameter $\delta\in (0,1)$, we consider an approximation of \eqref{scews} obtained by truncating the source term ($S_\delta:=S \mathbf{1}_{(0,1/\delta)}$) and adding an efflux term $-2\delta f$. We then show that the associated initial value problem is well-posed in $X_{0,1+\lambda}^+$ and construct an invariant set $\mathcal{Z}_\delta$, which is non-empty, convex, and sequentially weakly compact in $X_0$. Owing to these properties, an application of a consequence of Tychonov's fixed point theorem, see \cite[Theorem~1.2]{EMRR2005}, ensures the existence of a stationary solution $\varphi_\delta$  to this approximation. A by-product of the construction of the invariant set $\mathcal{Z}_\delta$ is the derivation of estimates which do not depend on the approximation parameter $\delta$ and ensure that the family $(\varphi)_{\delta\in (0,1)}$ lies in a sequentially weakly compact subset of $X_0$. We then show that the corresponding cluster points of $(\varphi)_{\delta\in (0,1)}$ as $\delta\to 0$ are stationary solutions to \eqref{scews1}, thereby completing the proof of Theorem~\ref{thmD0}~(a). We end up the paper with the non-existence of stationary solutions in the sense of Definition~\ref{defD1}, as stated in Theorem~\ref{thmD0}~(b), which is proved in Section~\ref{sec4}.

%%%%%%%%%%%%%%%%
%%%%%%%%%%%%%%%%
\section{Properties of stationary solutions}\label{sec2}
%%%%%%%%%%%%%%%%
%%%%%%%%%%%%%%%%

Let $\lambda\ge 0$ and consider a coagulation kernel $K$ satisfying \eqref{hypK} and $S\in X_0^+$. We first show that non-trivial stationary solutions to \eqref{scews1} do not decay too fast for large volumes, a property already observed in \cite[Theorem~8.1]{Dubo1994b} for $\lambda=0$ and in \cite[Corollary~6.4]{FLNV} for $\lambda\in [0,1)$. The proof given below differs from that in \cite{FLNV} and is closer to that in \cite{Dubo1994b}, an additional approximation argument being needed to handle the unboundedness of $K$ when $\lambda\in (0,1)$.

%%%%%%%%%%%%%%%%
\begin{proposition}\label{prc1}
	Assume that $\lambda\in [0,1)$ and let $\varphi$ be a stationary solution to \eqref{scews1}. If $S\not\equiv 0$, then $\varphi\not\in X_{(1+\lambda)/2}$.
\end{proposition}
%%%%%%%%%%%%%%%%

\begin{proof}
	We argue by contradiction and assume that $\varphi\in X_{(1+\lambda)/2}$. Then \begin{equation*}
	J(A) := \int_A^\infty x^{(1+\lambda)/2} \varphi(x)\ \mathrm{d}x 
	\end{equation*}
	is finite for all $A\ge 0$ and 
	\begin{equation*}
	\lim_{A\to \infty} J(A) = 0\,.
	\end{equation*}
	
	Now, let $A>0$ and set $\vartheta(x) := \min\{x, A\}$ for $x>0$. We infer from \eqref{D1} and the symmetry of $K$ that 
	\begin{equation}
	\begin{split}
	\int_0^\infty \vartheta_A(x) S(x)\ \mathrm{d}x & = \frac{1}{2} \int_0^A \int_{A-x}^A (x + y - A) K(x,y) \varphi(x) \varphi(y)\ \mathrm{d}y\mathrm{d}x \\ 
	& \quad + \int_0^A \int_A^\infty x K(x,y) \varphi(x) \varphi(y)\ \mathrm{d}y\mathrm{d}x \\
	& \quad + \frac{A}{2} \int_A^\infty \int_A^\infty  K(x,y) \varphi(x) \varphi(y)\ \mathrm{d}y\mathrm{d}x \,.
	\end{split} \label{c100}
	\end{equation} 
	
	We now study the behaviour of the terms on the right hand side of \eqref{c100} as $A\to\infty$. First, since $(1+\lambda)/2\in (0,1)$, it follows from \eqref{hypK} that, for $(x,y)\in (0,A)^2$ such that $x+y>A$,
	\begin{align*}
	(x+y-A) K(x,y) & \le k_2 (x+y-A) \left( x^\lambda + y^\lambda \right) \\
	& \le k_2 (x+y-A)^{(1-\lambda)/2} x^\lambda (x+y-A)^{(1+\lambda)/2} \\
	& \quad + k_2 (x+y-A)^{(1+\lambda)/2} (x+y-A)^{(1-\lambda)/2} y^\lambda \\
	& \le 2 k_2 (xy)^{(1+\lambda)/2}\,.
	\end{align*}
	Consequently,
	\begin{align*}
	& \frac{1}{2} \int_0^A \int_{A-x}^A (x + y - A) K(x,y) \varphi(x) \varphi(y)\ \mathrm{d}y\mathrm{d}x \\
	& \qquad \le k_2 \int_0^A \int_{A-x}^A (xy)^{(1+\lambda)/2} \varphi(x) \varphi(y)\ \mathrm{d}y\mathrm{d}x \\
	& \qquad \le k_2 \int_0^{A/2} \int_{A/2}^A (xy)^{(1+\lambda)/2} \varphi(x) \varphi(y)\ \mathrm{d}y\mathrm{d}x + k_2 \int_{A/2}^A \int_0^A (xy)^{(1+\lambda)/2} \varphi(x) \varphi(y)\ \mathrm{d}y\mathrm{d}x \\
	& \qquad \le 2 k_2 M_{(1+\lambda)/2}(\varphi) J(A/2)\,.
	\end{align*}
	Next, using again \eqref{hypK} and the property $\lambda\in [0,1)$, we find
	\begin{align*}
	\int_0^A \int_A^\infty x K(x,y) \varphi(x) \varphi(y)\ \mathrm{d}y\mathrm{d}x & \le k_2 \int_0^A \int_A^\infty \left( x^{1+\lambda} + x y^\lambda \right) \varphi(x) \varphi(y)\ \mathrm{d}y \mathrm{d}x \\
	& \le 2 k_2 \int_0^A \int_A^\infty (xy)^{(1+\lambda)/2} \varphi(x) \varphi(y)\ \mathrm{d}y \mathrm{d}x \\
	& \le 2 k_2 M_{(1+\lambda)/2}(\varphi) J(A)
	\end{align*}
	and 
	\begin{align*}
	\frac{A}{2} \int_A^\infty \int_A^\infty  K(x,y) \varphi(x) \varphi(y)\ \mathrm{d}y\mathrm{d}x & \le \frac{A k_2}{2} \int_A^\infty \int_A^\infty  \left( x^\lambda + y^\lambda \right) \varphi(x) \varphi(y)\ \mathrm{d}y\mathrm{d}x \\
	& \le A k_2 \int_A^\infty \int_A^\infty  x^\lambda \varphi(x) \varphi(y)\ \mathrm{d}y\mathrm{d}x \\
	& = k_2 \left( \int_A^\infty A^{(1-\lambda)/2} x^\lambda \varphi(x)\ \mathrm{d}x \right) \left( \int_A^\infty A^{(1+\lambda)/2} \varphi(y)\ \mathrm{d}y\mathrm{d}x \right) \\
	& \le k_2 J(A)^2\,.
	\end{align*}
	
	Gathering the above estimates, we deduce from \eqref{c100} that
	\begin{equation*}
	\int_0^\infty \min\{x, A\} S(x)\ \mathrm{d}x \le 2k_2 M_{(1+\lambda)/2}(\varphi) \left[ J(A/2) + J(A) \right] + k_2 J(A)^2\,.
	\end{equation*}
	Hence,
	\begin{equation*}
	\lim_{A\to \infty} \int_0^\infty \min\{x, A\} S(x)\ \mathrm{d}x =0\,,
	\end{equation*}
	which implies that $S\equiv 0$, and a contradiction.
\end{proof}

We next show that the behaviour of $S$ for small sizes governs that of stationary solutions.

%%%%%%%%%%%%%%%%
\begin{proposition}\label{prc2}
	Let $\varphi$ be a stationary solution to \eqref{scews1}. If $w\in C((0,\infty))$ is a non-negative and non-increasing function and $S\in L^1((0,\infty),w(x)\mathrm{d}x)$, then $\varphi\in L^1((0,\infty),w(x)\mathrm{d}x)$. In particular, if $S\in X_m$ for some $m\in (-\infty,0)$, then $\varphi\in X_m$.
\end{proposition}
%%%%%%%%%%%%%%%%

\begin{proof}
	Proposition~\ref{prc2} being obvious when $\varphi\equiv 0$, we may thus assume that $\varphi\not\equiv 0$. Consider $\varepsilon\in (0,1)$ and set $w_\varepsilon(x) := w(x+\varepsilon)$ for $x>0$. Owing to the monotonicity of $w$, there holds $w(x+\varepsilon)\le w(\varepsilon)$ for $x>0$ and
	\begin{equation*}
	-\chi_{w_\varepsilon}(x,y) = w(x+\varepsilon) + w(y+\varepsilon) - w(x+y+\varepsilon) \ge w(x+\varepsilon) \ge 0\,, \qquad (x,y)\in (0,\infty)^2\,.
	\end{equation*}
	We may then take $\vartheta=w_\varepsilon$ in \eqref{D1} and use the above inequality, the symmetry of $K$, and \eqref{hypK} to obtain
	\begin{align*}
	\int_0^\infty w_\varepsilon(x) S(x)\ \mathrm{d}x & = - \frac{1}{2} \int_0^\infty \int_0^\infty \chi_{w_\varepsilon}(x,y) K(x,y) \varphi(y) \varphi(x)\ \mathrm{d}y \mathrm{d}x \\
	& \ge k_1 \int_0^\infty \int_0^\infty \chi_{w_\varepsilon}(x,y) y^\lambda \varphi(y) \varphi(x)\ \mathrm{d}y \mathrm{d}x \\
	& \ge k_1 M_\lambda(\varphi) \int_0^\infty w(x+\varepsilon) \varphi(x)\ \mathrm{d}x\,.
	\end{align*}
	We then let $\varepsilon\to 0$ in the previous inequality and deduce from Fatou's lemma that
	\begin{equation*}
	\int_0^\infty w(x) S(x)\ \mathrm{d}x \ge k_1 M_\lambda(\varphi) \int_0^\infty w(x) \varphi(x)\ \mathrm{d}x\,,
	\end{equation*}
	thereby completing the proof, since $M_\lambda(\varphi)$ is finite and positive.
\end{proof}

We end up this section with upper and lower bounds on the moments of order zero and $\lambda$ of stationary solutions to \eqref{scews1}. 

%%%%%%%%%%%%%%%%
\begin{proposition}\label{prD2}
Let $\varphi$ be a stationary solution to \eqref{scews1}. Then
\begin{equation}
k_1 M_0(\varphi) M_\lambda(\varphi) \le M_0(S) \le k_2 M_0(\varphi) M_\lambda(\varphi)\,. \label{D2a}
\end{equation}
Assume further that $\lambda\in [0,1)$ and $S\in X_\lambda$. Then
\begin{equation}
\frac{2^\lambda}{k_2} M_\lambda(S) \le M_\lambda(\varphi)^2 \le \frac{2^{1-\lambda}}{k_1 (2-2^\lambda)} M_\lambda(S)\,. \label{D2b}
\end{equation}
\end{proposition}
%%%%%%%%%%%%%%%%

\begin{proof}
First, it follows from \eqref{D1} with the choice $\vartheta\equiv 1$ that
\begin{equation*}
M_0(S) = \frac{1}{2} \int_0^\infty \int_0^\infty K(x,y) \varphi(x) \varphi(y)\ \mathrm{d}y\mathrm{d}x\,.
\end{equation*}	
Combining the above identity with \eqref{hypK} readily gives \eqref{D2a}.

Next, the bounds \eqref{D2b} formally follow from \eqref{hypK} and \eqref{D1} with $\vartheta(x)=x^\lambda$, $x>0$. This function being not bounded, an approximation is needed. Specifically, let $A>0$ and set $\vartheta_A(x) = \min\{x^\lambda,A^\lambda\}$ for $x>0$. We infer from \eqref{D1} and the symmetry of $K$ that
\begin{equation}
\begin{split}
\int_0^\infty \vartheta_A(x) S(x)\ \mathrm{d}x & = \frac{1}{2} \int_0^A \int_0^{A-x} \left[ x^\lambda + y^\lambda - (x+y)^\lambda \right] K(x,y) \varphi(x) \varphi(y)\ \mathrm{d}y\mathrm{d}x \\
& \quad + \frac{1}{2} \int_0^A \int_{A-x}^A \left[ x^\lambda + y^\lambda - A^\lambda \right] K(x,y) \varphi(x) \varphi(y)\ \mathrm{d}y\mathrm{d}x \\ 
& \quad + \int_0^A \int_A^\infty x^\lambda K(x,y) \varphi(x) \varphi(y)\ \mathrm{d}y\mathrm{d}x \\
& \quad + \frac{A^\lambda}{2} \int_A^\infty \int_A^\infty  K(x,y) \varphi(x) \varphi(y)\ \mathrm{d}y\mathrm{d}x \,.
\end{split} \label{D3}
\end{equation} 

We now identify the limit as $A\to\infty$ of each term on the right hand side of \eqref{D3}. We first recall the following algebraic inequalities
\begin{equation}
2^\lambda (2-2^\lambda) \frac{(xy)^\lambda}{(x+y)^\lambda} \le x^\lambda + y^\lambda - (x+y)^\lambda \le \frac{(xy)^\lambda}{(x+y)^\lambda}\,, \qquad (x,y)\in (0,\infty)^2\,, \label{D4}
\end{equation}
see \cite[Eq.~(9)]{vDEr1985}, and
\begin{equation}
2^{\lambda-1} \left( x^\lambda + y^\lambda \right) \le (x+y)^\lambda \le x^\lambda + y^\lambda\,, \qquad (x,y)\in (0,\infty)^2\,, \label{D5} 
\end{equation}
which are valid due to $\lambda\in [0,1)$. We deduce from \eqref{hypK}, \eqref{D4}, and \eqref{D5} that
\begin{align*}
0 & \le \mathbf{1}_{(0,A)}(x) \mathbf{1}_{(0,A-x)}(y) \left[ x^\lambda + y^\lambda - (x+y)^\lambda \right] K(x,y) \varphi(x) \varphi(y) \\
& \le k_2 (xy)^\lambda \frac{x^\lambda+y^\lambda}{(x+y)^\lambda} \varphi(x) \varphi(y) \le 2^{1-\lambda} (xy)^\lambda \varphi(x)\varphi(y)\,.
\end{align*}
Since $\varphi\in X_\lambda$ and 
\begin{equation*}
\lim_{A\to\infty} \mathbf{1}_{(0,A)}(x) \mathbf{1}_{(0,A-x)}(y) = 1\,, \qquad (x,y)\in (0,\infty)^2\,,
\end{equation*}
Lebesgue's dominated convergence theorem entails that
\begin{align*}
& \lim_{A\to\infty} \frac{1}{2} \int_0^A \int_0^{A-x} \left[ x^\lambda + y^\lambda - (x+y)^\lambda \right] K(x,y) \varphi(x) \varphi(y)\ \mathrm{d}y\mathrm{d}x \\
& \qquad\qquad = \frac{1}{2} \int_0^\infty \int_0^\infty \left[ x^\lambda + y^\lambda - (x+y)^\lambda \right] K(x,y) \varphi(x) \varphi(y)\ \mathrm{d}y\mathrm{d}x\,.
\end{align*}
Next, by \eqref{hypK},
\begin{align*}
0 & \le \mathbf{1}_{(0,A)}(x) \mathbf{1}_{(A-x,A)}(y) \left[ x^\lambda + y^\lambda - A^\lambda \right] K(x,y) \varphi(x) \varphi(y) \\
& \le k_2  \left[ x^\lambda + y^\lambda - A^\lambda \right] \left( x^\lambda+y^\lambda \right) \varphi(x) \varphi(y) \\
& \le k_2 \left( x^\lambda y^\lambda + y^\lambda x^\lambda \right) \varphi(x)\varphi(y) = 2 k_2 (xy)^\lambda \varphi(x) \varphi(y)\,.
\end{align*}
Since $\varphi\in X_\lambda$ and 
\begin{equation*}
\lim_{A\to\infty} \mathbf{1}_{(0,A)}(x) \mathbf{1}_{(A-x,A)}(y) = 0\,, \qquad (x,y)\in (0,\infty)^2\,,
\end{equation*}
we use again Lebesgue's dominated convergence theorem to obtain
\begin{equation*}
\lim_{A\to\infty} \frac{1}{2} \int_0^A \int_{A-x}^A \left[ x^\lambda + y^\lambda - A^\lambda \right] K(x,y) \varphi(x) \varphi(y)\ \mathrm{d}y\mathrm{d}x = 0\,.
\end{equation*}
Finally, using once more \eqref{hypK},
\begin{align*}
0 & \le \int_0^A \int_A^\infty x^\lambda K(x,y) \varphi(x) \varphi(y)\ \mathrm{d}y\mathrm{d}x \le k_2 \int_0^A \int_A^\infty x^\lambda \left( x^\lambda + y^\lambda \right) \varphi(x) \varphi(y)\ \mathrm{d}y\mathrm{d}x \\
& \le 2 k_2 \int_0^A \int_A^\infty (xy)^\lambda \varphi(x) \varphi(y)\ \mathrm{d}y\mathrm{d}x \le 2 k_2 M_\lambda(\varphi) \int_A^\infty y^\lambda \varphi(y)\ \mathrm{d}y
\end{align*}
and
\begin{align*}
0 & \le A^\lambda \int_A^\infty \int_A^\infty K(x,y) \varphi(x) \varphi(y)\ \mathrm{d}y\mathrm{d}x \le k_2 A^\lambda \int_A^\infty \int_A^\infty \left( x^\lambda + y^\lambda \right) \varphi(x) \varphi(y)\ \mathrm{d}y\mathrm{d}x \\
& \le 2 k_2 \int_A^\infty \int_A^\infty (xy)^\lambda \varphi(x) \varphi(y)\ \mathrm{d}y\mathrm{d}x \le 2 k_2 M_\lambda(\varphi) \int_A^\infty y^\lambda \varphi(y)\ \mathrm{d}y\,,
\end{align*}
from which we deduce that
\begin{align*}
& \lim_{A\to\infty} \int_0^A \int_A^\infty x^\lambda K(x,y) \varphi(x) \varphi(y)\ \mathrm{d}y\mathrm{d}x = 0\,, \\
& \lim_{A\to\infty} \frac{A^\lambda}{2} \int_A^\infty \int_A^\infty  K(x,y) \varphi(x) \varphi(y)\ \mathrm{d}y\mathrm{d}x = 0\,,
\end{align*}
recalling that $\varphi\in X_\lambda$. Collecting the above information, we may take the limit $A\to\infty$ in \eqref{D3} and obtain, since $S\in X_\lambda$, 
\begin{equation}
M_\lambda(S) = \frac{1}{2} \int_0^\infty \int_0^\infty \left[ x^\lambda + y^\lambda - (x+y)^\lambda \right] K(x,y) \varphi(x) \varphi(y)\ \mathrm{d}y\mathrm{d}x\,. \label{D6}
\end{equation}

Now, we infer from \eqref{hypK}, \eqref{D4}, \eqref{D5}, and \eqref{D6} that \begin{equation*}
2^\lambda (1-2^{\lambda-1}) k_1 M_\lambda(\varphi)^2 \le M_\lambda(S) \le 2^{-\lambda} k_2 M_\lambda(\varphi)^2\,,
\end{equation*}
from which \eqref{D2b} follows.
\end{proof}

%%%%%%%%%%%%%%%%
%%%%%%%%%%%%%%%%
\section{Approximation}\label{sec3}
%%%%%%%%%%%%%%%%
%%%%%%%%%%%%%%%%

Throughout this section, we assume that $\lambda\in [0,1)$ and that the coagulation kernel $K$ satisfies \eqref{hypK} and \eqref{monK}. Also, let $S$ be a source term satisfying 
\begin{equation}
S\in \bigcap_{m\in [0,1)} X_m^+\,, \qquad S\not\equiv 0\,. \label{b2}
\end{equation}
Since $S\in X_0$, it follows from a refined version of the de la Vall\'ee-Poussin theorem \cite{dlVP15}, see \cite{Le1977} or \cite[Theorem~7.1.6]{BLL2019}, that there is a function $\Phi\in C^1([0,\infty))$ depending only on $S$ which satisfies the following properties: $\Phi$ is convex, $\Phi(0)=\Phi'(0)=0$, $\Phi'$ is a concave function which is positive on $(0,\infty)$,
\begin{subequations}\label{v0}
\begin{equation}
\lim_{r\to\infty} \Phi'(r) = \lim_{r\to\infty} \frac{\Phi(r)}{r} = \infty\,, \label{v1}
\end{equation}
and
\begin{equation}
L_\Phi(S) := \int_0^\infty \Phi(S(x))\ \mathrm{d}x < \infty\,. \label{v2}
\end{equation}
\end{subequations}

For $\delta\in (0,1)$, we define 
\begin{equation}
S_\delta = S \mathbf{1}_{(0,1/\delta)} \in X_{0,1+\lambda}^+\,. \label{n1}
\end{equation}

We shall then prove the existence of a stationary solution to the following approximation of \eqref{scews}
\begin{subequations}\label{b4}
\begin{align}
\partial_t f(t,x) & = \mathcal{C} f(t,x) + S_\delta(x) - 2 \delta f(t,x)\,, \qquad (t,x)\in (0,\infty)^2\,, \label{b4a} \\
f(0,x) & = f^{in}(x)\,, \qquad x\in (0,\infty)\,, \label{b4b}
\end{align}
\end{subequations}
 which is a coagulation equation with a truncated source term and an additional efflux term. In \eqref{b4a}, the coagulation operator $\mathcal{C}f$ is still given by \eqref{scews3}.

%%%%%%%%%%%%%%%%
\begin{proposition}\label{prb0}
There is $\delta_0\in (0,1)$  depending only on $\lambda$, $k_1$, $k_2$, and $S$ such that, for $\delta\in (0,\delta_0)$, there is at least one stationary solution $\varphi_\delta\in X_{0,1+\lambda}^+$ to \eqref{b4a} which satisfies the following properties: for all $\vartheta\in L^\infty(0,\infty)$,
\begin{equation}
\frac{1}{2} \int_0^\infty \int_0^\infty \chi_\vartheta(x,y) K(x,y) \varphi_\delta(x)\varphi_\delta(y)\ \mathrm{d}y \mathrm{d}x + \int_0^\infty S_\delta(x) \vartheta(x)\ \mathrm{d}x = 2\delta \int_0^\infty \varphi_\delta(x) \vartheta(x)\ \mathrm{d}x\,, \label{n3}
\end{equation}
the function $\chi_\vartheta$ being defined in \eqref{D1a}, and there are positive constants $\gamma_1>0$ and $\gamma_2>0$ depending only on $\lambda$, $k_1$, $k_2$, and $S$ such that
\begin{equation}
0 < \gamma_1 \le M_\lambda(\varphi_\delta) \le \gamma_2\,, \qquad \int_0^\infty \Phi(\varphi_\delta(x))\ \mathrm{d}x \le \gamma_2\,, \label{c4}
\end{equation}
and, for each $\mu\in [0,(1+\lambda)/2)$, there is a positive constant $\gamma_3(\mu)>0$ depending only on $\lambda$, $k_1$, $k_2$, $S$, and $\mu$ such that
\begin{equation}
M_\mu(\varphi_\delta) \le \gamma_3(\mu)\,. \label{c5}
\end{equation}
\end{proposition}
%%%%%%%%%%%%%%%%

As in \cite{FLNV}, the proof of Proposition~\ref{prb0} relies on a dynamical approach. As already outlined in the Introduction, it amounts  to prove that the coagulation equation with source and efflux terms \eqref{b4} is well-posed in an appropriately defined subset of $X_0$, which is here chosen to be $X_{0,1+\lambda}^+$, and generates a semi-flow $\Psi_\delta(\cdot,f^{in})$ on that set endowed with the weak topology of $X_0$, while leaving invariant a closed convex and weakly compact subset $\mathcal{Z}_\delta$. We then deduce from an application of Tychonov's fixed point theorem, see \cite[Theorem~1.2]{EMRR2005}, that the semi-flow $\Psi_\delta$ has a fixed point in $\mathcal{Z}_\delta$, which is obviously a stationary solution to \eqref{b4a}. To set up the stage for the proof of Proposition~\ref{prb0}, we first state the well-posedness of \eqref{b4} in $X_{0,1+\lambda}^+$.

%%%%%%%%%%%%%%%%
\begin{proposition}\label{prb00}
Let $\delta\in (0,1)$. Given $f^{in}\in X_{0,1+\lambda}^+$, there is a unique weak solution $f_\delta = \Psi_\delta(\cdot,f^{in})$ to \eqref{b4} satisfying
\begin{align*}
f_\delta & \in C([0,\infty),X_0^+)\,, \qquad f_\delta(0) = f^{in}\,, \\
f_\delta & \in W^{1,\infty}((0,T),X_0) \cap L^\infty((0,T),X_{1+\lambda})\,, \qquad T>0\,,
\end{align*}
and
\begin{equation}
\begin{split}
\frac{\mathrm{d}}{\mathrm{d}t} \int_0^\infty f_\delta(t,x) \vartheta(x)\ \mathrm{d}x & = \frac{1}{2} \int_0^\infty \int_0^\infty \chi_\vartheta(x,y) K(x,y) f_\delta(t,x) f_\delta(t,y)\ \mathrm{d}y \mathrm{d}x \\
& \qquad + \int_0^\infty S_\delta(x) \vartheta(x)\ \mathrm{d}x - 2\delta \int_0^\infty f_\delta(t,x) \vartheta(x)\ \mathrm{d}x
\end{split} \label{wFd}
\end{equation}
for all $t>0$ and $\vartheta\in L^\infty(0,\infty)$. Moreover, if $R>0$ and $(f_j^{in})_{j\ge 1}$ is a sequence in $\{ h\in X_{0,1+\lambda}^+\ :\ M_{1+\lambda}(h) \le R\}$ which converges weakly in $X_0$ to $f^{in}$, then $(\Psi_\delta(\cdot,f_j^{in}))_{j\ge 1}$ converges to $\Psi_\delta(\cdot,f^{in})$ in $C([0,T],X_{0,w})$ for any $T>0$, where $X_{0,w}$ denotes the space $X_0$ endowed with its weak topology.
\end{proposition}
%%%%%%%%%%%%%%%%

 Since the proof of Proposition~\ref{prb00} follows the same lines as that of similar results for coagulation-fragmentation equations and stronger versions of most of the estimates involved in it are derived in Sections~\ref{sec3.1}-\ref{sec3.2}, we omit the proof here and refer instead to \cite{BLL2019, Dubo1994b, EMRR2005, Stew1989} and the references therein. Let us also mention here that the well-posedness of the discrete coagulation-fragmentation equations with source and efflux terms is specifically studied in \cite{KuTh2019, Spou1985b}.

\bigskip

\newcounter{NumConst}

In the following, $C$ and $(C_i)_{i\ge 1}$ denote positive constant depending only on $\lambda$, $k_1$, $k_2$, and $S$. Dependence upon additional parameters will be indicated explicitly. Also, for $m\in\mathbb{R}$ and $x\in (0,\infty)$, we set $\vartheta_m(x) := x^m$ and $\chi_m := \chi_{\vartheta_m}$.

%%%%%%%%%%%%%%%%
%%%%%%%%%%%%%%%%
\subsection{Moment estimates}\label{sec3.1}
%%%%%%%%%%%%%%%%
%%%%%%%%%%%%%%%%

We begin with a bound on the moment of order $\lambda$ which depends, neither on $\delta\in (0,1)$, nor on $t>0$.

\refstepcounter{NumConst}\label{cst1}

%%%%%%%%%%%%%%%%
\begin{lemma}\label{lemb1}
There is $C_{\ref{cst1}}>0$ such that, if 
\begin{equation}
M_\lambda(f^{in})\le C_{\ref{cst1}} := \sqrt{\frac{2 M_\lambda(S)}{(1-2^{\lambda-1})k_1}}\,, \label{inv1}
\end{equation} 
then 
\begin{equation*}
M_\lambda(f_\delta(t)) \le C_{\ref{cst1}}\,, \qquad t\ge 0\,.
\end{equation*} 
\end{lemma}
%%%%%%%%%%%%%%%%

\begin{proof}
Let $t>0$. It follows from \eqref{wFd} that
\begin{equation*}
\frac{\mathrm{d}}{\mathrm{d}t} M_\lambda(f_\delta(t)) - \frac{1}{2} \int_0^\infty \int_0^\infty \chi_{\lambda}(x,y) K(x,y) f_\delta(t,x) f_\delta(t,y)\ \mathrm{d}y \mathrm{d}x = M_\lambda(S_\delta) - 2\delta M_\lambda(f_\delta(t))\,.
\end{equation*}
Arguing as in \cite[Lemma~3.1, Step~1]{FoLa2005}, we infer from \eqref{hypK} and the symmetry of $K_\delta$ that
\begin{align*}
& \frac{1}{2} \int_0^\infty \int_0^\infty \left[ x^\lambda + y^\lambda - (x+y)^\lambda \right] K(x,y) f_\delta(t,x) f_\delta(t,y)\ \mathrm{d}y \mathrm{d}x \\
& \qquad = \int_0^\infty \int_0^\infty x \left[ x^{\lambda-1} - (x+y)^{\lambda-1} \right] K(x,y) f_\delta(t,x) f_\delta(t,y)\ \mathrm{d}y \mathrm{d}x \\
& \qquad \ge k_1\int_0^\infty \int_0^\infty xy^\lambda \left[ x^{\lambda-1} - (x+y)^{\lambda-1} \right] f_\delta(t,x) f_\delta(t,y)\ \mathrm{d}y \mathrm{d}x \\
& \qquad \ge k_1 \int_0^\infty \int_x^\infty x y^\lambda \left[ x^{\lambda-1} - (2x)^{\lambda-1} \right] f_\delta(t,x) f_\delta(t,y)\ \mathrm{d}y \mathrm{d}x \\
& \qquad \ge (1-2^{\lambda-1}) k_1 \int_0^\infty \int_x^\infty x^\lambda y^\lambda f_\delta(t,x) f_\delta(t,y)\ \mathrm{d}y \mathrm{d}x \\
& \qquad \ge \frac{(1-2^{\lambda-1}) k_1}{2} M_\lambda(f_\delta(t))^2\,.
\end{align*}
Consequently, using also \eqref{n1},
\begin{equation*}
\frac{\mathrm{d}}{\mathrm{d}t} M_\lambda(f_\delta(t)) + \frac{(1-2^{\lambda-1}) k_1}{2} M_\lambda(f_\delta(t))^2 \le M_\lambda(S)\,,
\end{equation*}
from which we deduce by the comparison principle that
\begin{equation*}
M_\lambda(f_\delta(t)) \le \max\left\{ M_\lambda(f^{in}) , C_{\ref{cst1}} \right\}\,, \qquad t\ge 0\,.
\end{equation*}
Lemma~\ref{lemb1} then follows, thanks to \eqref{inv1}.
\end{proof}

The next step is the derivation of two bounds on the moment of order zero, the first one depending on $\delta\in (0,1)$ but not on $t>0$, while the second one depends mildly on $t>0$ but not on $\delta\in (0,1)$.

%%%%%%%%%%%%%%%%
\begin{lemma}\label{lemb2}	
If 
\begin{equation}
M_0(f^{in}) \le \frac{M_0(S)}{2\delta}\,, \label{inv2}
\end{equation} 
then 
\begin{equation}
M_0(f_\delta(t)) \le \frac{M_0(S)}{2\delta}\,, \qquad t\ge 0\,. \label{b5} 
\end{equation}
In addition,
\begin{equation}
\frac{k_1}{t} \int_0^t M_0(f_\delta(s)) M_\lambda(f_\delta(s))\ \mathrm{d}s  \le \frac{M_0(f^{in})}{t} + M_0(S)\,, \qquad t>0\,. \label{b6}
\end{equation} 
\end{lemma}
%%%%%%%%%%%%%%%%

\begin{proof}
Let $t>0$. By \eqref{wFd},
\begin{equation*}
\frac{\mathrm{d}}{\mathrm{d}t} M_0(f_\delta(t)) + \frac{1}{2} \int_0^\infty \int_0^\infty K(x,y) f_\delta(t,x) f_\delta(t,y)\ \mathrm{d}y \mathrm{d}x = M_0(S_\delta) - 2\delta M_0(f_\delta(t))\,,
\end{equation*}
which entails, together with \eqref{hypK} and \eqref{n1}, that
\begin{equation}
\frac{\mathrm{d}}{\mathrm{d}t} M_0(f_\delta(t)) + 2\delta M_0(f_\delta(t)) + k_1 M_0(f_\delta(t)) M_\lambda(f_\delta(t)) \le M_0(S)\,. \label{b7}
\end{equation}
It first follows from \eqref{b7} that 
\begin{equation*}
\frac{\mathrm{d}}{\mathrm{d}t} M_0(f_\delta(t)) + 2\delta M_0(f_\delta(t)) \le M_0(S)\,.
\end{equation*}
Hence, 
\begin{equation*}
M_0(f_\delta(t)) \le e^{-2\delta t} M_0(f^{in}) + \frac{M_0(S)}{2\delta} (1-e^{-2\delta t}) \le \max\left\{ M_0(f^{in}) , \frac{M_0(S)}{2\delta} \right\}
\end{equation*}
from which we deduce \eqref{b5} after using \eqref{inv2}. We next integrate \eqref{b7} with respect to time over $(0,t)$ and discard the first two non-negative terms in the left hand side of the resulting inequality divided by $t$ to obtain \eqref{b6}.
\end{proof}

We now derive bounds for moments of order up to $(1+\lambda)/2$. To this end, we need the following lemma.

%%%%%%%%%%%%%%%%
\begin{lemma}\label{lemb3}
Consider $\theta\in [0,1/2]$, $m\in (0,1)$, and $\sigma\in [0,(m+2\theta)/2)$. If $g\in L^1((1,\infty),x^\sigma\mathrm{d}x)$ is non-negative almost everywhere in $(1,\infty)$, then
\begin{equation*}
\left( \int_1^\infty x^\sigma g(x)\ \mathrm{d}x \right)^2 \le \frac{\kappa(\theta,m,\sigma)}{2} \int_1^\infty \int_1^\infty \left[ x^m + y^m - (x+y)^m \right] (xy)^\theta g(x) g(y)\ \mathrm{d}y \mathrm{d}x\,,
\end{equation*}
where
\begin{equation*}
\kappa(\theta,m,\sigma) := \frac{ 2^{1-m}\pi^2}{3(1-m)} 4^{(2-m)/(m+2\theta-2\sigma)}\,.
\end{equation*}
\end{lemma}
%%%%%%%%%%%%%%%%

\begin{proof}
We argue as in the proof of \cite[Lemma~8.2.14]{BLL2019} and define $\zeta := 2/(m+2\theta-2\sigma)>0$ and $x_i := i^\zeta$, $i\ge 1$. Setting
\begin{equation*}
\mathcal{I} := \frac{1}{2} \int_1^\infty \int_1^\infty \left[ x^m + y^m - (x+y)^m \right] (xy)^\theta g(x) g(y)\ \mathrm{d}y \mathrm{d}x \ge 0\,,
\end{equation*}
Lemma~\ref{lemb3} is obviously true if $\mathcal{I}=\infty$. We then assume that $\mathcal{I}<\infty$ and observe that
\begin{align}
\mathcal{I} & = \int_1^\infty \int_1^\infty \left[ x^{m-1} - (x+y)^{m-1} \right] x^{\theta+1} y^\theta g(x) g(y)\ \mathrm{d}y \mathrm{d}x \nonumber \\
& \ge (1-m) \int_1^\infty \int_1^\infty (x+y)^{m-2} (xy)^{\theta+1} g(x) g(y)\ \mathrm{d}y \mathrm{d}x \nonumber \\
& \ge (1-m) \sum_{i=1}^\infty \int_{x_i}^{x_{i+1}} \int_{x_i}^{x_{i+1}} (x+y)^{m-2} (xy)^{\theta+1} g(x) g(y)\ \mathrm{d}y \mathrm{d}x \nonumber \\
& \ge (1-m) 2^{m-2} \sum_{i=1}^\infty x_{i+1}^{m-2} \int_{x_i}^{x_{i+1}} \int_{x_i}^{x_{i+1}} (xy)^{\theta+1} g(x) g(y)\ \mathrm{d}y \mathrm{d}x \nonumber \\
& = (1-m) 2^{m-2} \sum_{i=1}^\infty x_{i+1}^{m-2} J_i^2\,, \label{b8}
\end{align}
where
\begin{equation*}
J_i := \int_{x_i}^{x_{i+1}} x^{\theta+1} g(x)\ \mathrm{d}x\,, \qquad i\ge 1\,.
\end{equation*}

Next, since $\sigma<1+\theta$, we infer from the Cauchy-Schwarz inequality that
\begin{align*}
\int_1^\infty x^\sigma g(x)\ \mathrm{d}x & = \sum_{i=1}^\infty \int_{x_i}^{x_{i+1}} x^\sigma g(x)\ \mathrm{d}x \le \sum_{i=1}^\infty x_i^{\sigma-1-\theta} J_i \\
& \le \left( \sum_{i=1}^\infty x_i^{2\sigma - 2 - 2\theta} x_{i+1}^{2-m} \right)^{1/2} \left( \sum_{i=1}^\infty x_{i+1}^{m-2} J_i^2 \right)^{1/2}\,.
\end{align*}
Hence,
\begin{equation}
\left( \int_1^\infty x^\sigma g(x)\ \mathrm{d}x \right)^2 \le \left( \sum_{i=1}^\infty x_i^{2\sigma - 2 - 2\theta} x_{i+1}^{2-m} \right) \sum_{i=1}^\infty x_{i+1}^{m-2} J_i^2\,. \label{b9}
\end{equation}
Owing to the definition of $(x_i)_{i\ge 1}$ and $\zeta$, 
\begin{align*}
\sum_{i=1}^\infty x_i^{2\sigma - 2 - 2\theta} x_{i+1}^{2-m} & \le \sum_{i=1}^\infty i^{(2\sigma - 2 - 2\theta)\zeta} (2i)^{(2-m)\zeta} = 2^{(2-m)\zeta} \frac{\pi^2}{6}\,.
\end{align*}
Combining \eqref{b8} and \eqref{b9} gives
\begin{equation*}
\mathcal{I} \ge \frac{6(1-m) 2^{m-2}}{2^{(2-m)\zeta} \pi^2} \left( \int_1^\infty x^\sigma g(x)\ \mathrm{d}x \right)^2\,,
\end{equation*}
as claimed.
\end{proof}

Thanks to Lemma~\ref{lemb3}, we are now in a position to estimate moments of order $m\in (0,1)$. As in Lemma~\ref{lemb2}, two estimates are derived, one which depends on $\delta\in (0,1)$ but not on $t>0$, the other one being independent of $\delta\in (0,1)$ with a mild dependence upon $t>0$.
\refstepcounter{NumConst}\label{cst2}

%%%%%%%%%%%%%%%%
\begin{lemma}\label{lemb4}
Let $m\in (0,1)$ and $\mu\in [0,(m+\lambda)/2)$. If 
\begin{equation}
M_m(f^{in})\le \frac{M_m(S)}{2\delta}\,, \label{inv3}
\end{equation} 
then 
\begin{equation}
M_m(f_\delta(t)) \le \frac{M_m(S)}{2\delta}\,, \qquad t\ge 0\,. \label{b10}
\end{equation}
Moreover, there is $C_{\ref{cst2}}(m,\mu)>0$ such that
\begin{equation}
\frac{1}{t} \int_0^t \left( \int_1^\infty x^\mu f_\delta(s,x)\ \mathrm{d}x \right)^2  \mathrm{d}s \le C_{\ref{cst2}}(m,\mu) \left( \frac{M_m(f^{in})}{t} + M_m(S) \right) \,. \label{b11}
\end{equation}
\end{lemma}
%%%%%%%%%%%%%%%%

\begin{proof}
Let $t>0$. By \eqref{wFd}, 
\begin{equation}
\frac{\mathrm{d}}{\mathrm{d}t} M_m(f_\delta(t)) - \frac{1}{2} \int_0^\infty \int_0^\infty \chi_{m}(x,y) K(x,y) f_\delta(t,x) f_\delta(t,y)\ \mathrm{d}y \mathrm{d}x = M_m(S_\delta) - 2\delta M_m(f_\delta(t))\,. \label{b12}
\end{equation}
We infer from \eqref{hypK}, the inequality $x^\lambda+y^\lambda \ge 2 (xy)^{\lambda/2}$, $(x,y)\in (0,\infty)^2$, and Lemma~\ref{lemb3} (with $(\theta,m,\sigma)=(\lambda/2,m,\mu)$) that
\begin{align*}
& \frac{1}{2} \int_0^\infty \int_0^\infty \left[ x^m + y^m - (x+y)^m \right] K(x,y) f_\delta(t,x) f_\delta(t,y)\ \mathrm{d}y \mathrm{d}x \\ 
& \hspace{3cm} \ge k_1 \int_0^\infty \int_0^\infty \left[ x^m + y^m - (x+y)^m \right] (xy)^{\lambda/2} f_\delta(t,x) f_\delta(t,y)\ \mathrm{d}y \mathrm{d}x \\
& \hspace{3cm} \ge k_1 \int_1^\infty \int_1^\infty \left[ x^m + y^m - (x+y)^m \right] (xy)^{\lambda/2} f_\delta(t,x) f_\delta(t,y)\ \mathrm{d}y \mathrm{d}x \\
& \hspace{3cm} \ge \frac{2k_1}{\kappa(\lambda/2,m,\mu)} \left( \int_1^\infty x^\mu f_\delta(t,x)\ \mathrm{d}x \right)^2\,.
\end{align*}
Consequently, along with \eqref{n1}, we obtain
\begin{equation}
\frac{\mathrm{d}}{\mathrm{d}t} M_m(f_\delta(t)) + \frac{1}{C_{\ref{cst2}}(m,\mu)} \left( \int_1^\infty x^\mu f_\delta(t,x)\ \mathrm{d}x \right)^2 +  2\delta M_m(f_\delta(t)) \le M_m(S)\,. \label{b13}
\end{equation}
A first consequence of \eqref{b13} is that 
\begin{equation*}
\frac{\mathrm{d}}{\mathrm{d}t} M_m(f_\delta(t)) +  2\delta M_m(f_\delta(t)) \le M_m(S)\,.
\end{equation*}
After integration, we obtain
\begin{equation*}
M_m(f_\delta(t)) \le e^{-2\delta t} M_m(f^{in}) + \frac{M_m(S)}{2\delta} (1-e^{-2\delta t}) \le \max\left\{ M_m(f^{in}) , \frac{M_m(S)}{2\delta} \right\}
\end{equation*}
and use \eqref{inv3} to deduce \eqref{b10}. We next integrate \eqref{b13} with respect to time over $(0,t)$ and discard the non-negative contributions of the first and third terms in the left hand side of the resulting inequality divided by $t$ to obtain \eqref{b11}.
\end{proof}

We next derive estimates in $X_1\cap X_{1+\lambda}$ which strongly depend on $\delta$. \refstepcounter{NumConst}\label{cst3}

%%%%%%%%%%%%%%%%
\begin{lemma}\label{lemn1}
There is $C_{\ref{cst3}}>0$ such that, if $f^{in}$ satisfies \eqref{inv2} along with
\begin{equation}
M_1(f^{in}) \le \frac{M_\lambda(S)}{2\delta^{2-\lambda}} \;\;\text{ and }\;\; M_{1+\lambda}(f^{in}) \le \frac{C_{\ref{cst3}}}{\delta^{(4+\lambda-\lambda^2)/(1-\lambda)}}\,, \label{inv4}
\end{equation}
then
\begin{equation*}
M_1(f_\delta(t)) \le \frac{M_\lambda(S)}{2\delta^{2-\lambda}} \;\;\text{ and }\;\; M_{1+\lambda}(f_\delta(t)) \le \frac{C_{\ref{cst3}}}{\delta^{(4+\lambda-\lambda^2)/(1-\lambda)}}\,, \qquad t\ge 0\,.
\end{equation*}	
\end{lemma}
%%%%%%%%%%%%%%%%

\begin{proof}
Let $t>0$. It first follows from \eqref{n1} and \eqref{wFd} that
\begin{equation*}
\frac{\mathrm{d}}{\mathrm{d}t} M_1(f_\delta(t)) + 2\delta M_1(f_\delta(t)) = M_1(S_\delta) \le \frac{M_\lambda(S)}{\delta^{1-\lambda}}\,.
\end{equation*}
Hence,
\begin{equation*}
M_1(f_\delta(t)) \le e^{-2\delta t} M_1(f^{in}) + \frac{M_\lambda(S)}{2\delta^{2-\lambda}} (1-e^{-2\delta t}) \le \max\left\{ M_1(f^{in}) , \frac{M_\lambda(S)}{2\delta^{2-\lambda}} \right\}\,,
\end{equation*}
which, together with \eqref{inv4}, readily gives the claimed estimate on $M_1(f_\delta)$. We next infer from \eqref{hypK}, \eqref{n1}, and \eqref{wFd} that
\begin{align*}
& \frac{\mathrm{d}}{\mathrm{d}t} M_{1+\lambda}(f_\delta(t)) + 2\delta M_{1+\lambda}(f_\delta(t)) \\
& \qquad = \frac{1}{2} \int_0^\infty \int_0^\infty \chi_{1+\lambda}(x,y) K(x,y) f_\delta(t,x) f_\delta(t,y)\ \mathrm{d}y \mathrm{d}x + M_{1+\lambda}(S_\delta) \\
& \qquad \le \frac{k_2}{2} \int_0^\infty \int_0^\infty \chi_{1+\lambda}(x,y) (x^\lambda + y^\lambda) f_\delta(t,x) f_\delta(t,y)\ \mathrm{d}y \mathrm{d}x + \frac{M_\lambda(S)}{\delta} \\
& \qquad = k_2 \int_0^\infty \int_0^\infty x^\lambda \chi_{1+\lambda}(x,y)  f_\delta(t,x) f_\delta(t,y)\ \mathrm{d}y \mathrm{d}x + \frac{M_\lambda(S)}{\delta}\,.
\end{align*}
For $(x,y)\in (0,\infty)^2$, it follows from \cite[Lemma~7.4.4]{BLL2019} that
\begin{equation*}
\chi_{1+\lambda}(x,y) = (x+y)^{1+\lambda} - x^{1+\lambda} - y^{1+\lambda} \le (1+\lambda) \frac{x^{1+\lambda}y + x y^{1+\lambda}}{x+y}\,, 
\end{equation*}
from which we deduce that
\begin{align*}
x^\lambda \chi_{1+\lambda}(x,y) & \le (1+\lambda) \frac{x^{1+2\lambda}y + x^{1+\lambda} y^{1+\lambda}}{x+y} \\ 
&\le (1+\lambda) \left[ \frac{x}{x+y} + \frac{x^{1-\lambda}}{(x+y)^{1-\lambda}} \frac{y^\lambda}{(x+y)^\lambda} \right] x^{2\lambda} y \\
& \le 4 x^{2\lambda}y\,.
\end{align*}
Therefore,
\begin{equation*}
\frac{\mathrm{d}}{\mathrm{d}t} M_{1+\lambda}(f_\delta(t)) + 2\delta M_{1+\lambda}(f_\delta(t)) \le 4 k_2 M_{2\lambda}(f_\delta(t)) M_1(f_\delta(t))\,,
\end{equation*}
and we use the just established bound on $M_1(f_\delta(t))$ to obtain
 \begin{equation*}
 \frac{\mathrm{d}}{\mathrm{d}t} M_{1+\lambda}(f_\delta(t)) + 2\delta M_{1+\lambda}(f_\delta(t)) \le 2 k_2 \frac{M_\lambda(S)}{\delta^{2-\lambda}} M_{2\lambda}(f_\delta(t)) \,. 
 \end{equation*}
Now, since $2\lambda\in [0,1+\lambda)$, it follows from \eqref{b5} and H\"older's inequality that
\begin{align*}
M_{2\lambda}(f_\delta(t)) & \le M_{1+\lambda}(f_\delta(t))^{2\lambda/(1+\lambda)} M_0(f_\delta(t))^{(1-\lambda)/(1+\lambda)} \\
& \le \left( \frac{M_0(S)}{2\delta} \right)^{(1-\lambda)/(1+\lambda)} M_{1+\lambda}(f_\delta(t))^{2\lambda/(1+\lambda)}\,.
\end{align*}
Combining the above two inequalities gives
\begin{equation*}
 \frac{\mathrm{d}}{\mathrm{d}t} M_{1+\lambda}(f_\delta(t)) + 2\delta M_{1+\lambda}(f_\delta(t)) \le C_{\ref{cst3}}^{(1-\lambda)/(1+\lambda)} \delta^{-(3-\lambda^2)/(1+\lambda)} M_{1+\lambda}(f_\delta(t))^{2\lambda/(1+\lambda)}\,,
\end{equation*}
with
\begin{equation*}
C_{\ref{cst3}} := (2k_2 M_\lambda(S))^{(1+\lambda)/(1-\lambda)} \frac{M_0(S)}{2}\,.
\end{equation*}
We finally use Young's inequality to derive 
\begin{equation*}
\frac{\mathrm{d}}{\mathrm{d}t} M_{1+\lambda}(f_\delta(t)) + 2\delta M_{1+\lambda}(f_\delta(t)) \le \delta M_{1+\lambda}(f_\delta(t)) + C_{\ref{cst3}} \delta^{-(2\lambda+3-\lambda^2)/(1-\lambda)}\,.
\end{equation*}
Hence,
\begin{equation*}
\frac{\mathrm{d}}{\mathrm{d}t} M_{1+\lambda}(f_\delta(t)) + \delta M_{1+\lambda}(f_\delta(t)) \le C_{\ref{cst3}} \delta^{-(2\lambda+3-\lambda^2)/(1-\lambda)}\,,
\end{equation*}
from which we deduce
\begin{align*}
M_{1+\lambda}(f_\delta(t)) & \le e^{-\delta t} M_{1+\lambda}(f^{in}) + \frac{C_{\ref{cst3}}}{\delta^{(4+\lambda-\lambda^2)/(1-\lambda)}} (1-e^{-\delta t}) \\
& \le \max\left\{ M_{1+\lambda}(f^{in}) , \frac{C_{\ref{cst3}}}{\delta^{(4+\lambda-\lambda^2)/(1-\lambda)}} \right\}\,.
\end{align*}
Combining \eqref{inv4} with the above inequality completes the proof.
\end{proof}

We end up this section with a lower bound on the moment of order $\lambda$ in the spirit of that established in Proposition~\ref{prD2} which depends, neither on $\delta\in (0,1)$, nor on $t>0$, provided the former is small enough.

%%%%%%%%%%%%%%%%
\begin{lemma}\label{lemb4b}
There are $C_{\ref{cst4}}>0$ and $\delta_0\in (0,1)$ depending only on $\lambda$, $k_1$, $k_2$, and $S$ and such that, if $\delta\in (0,\delta_0)$ and \refstepcounter{NumConst}\label{cst4}
\begin{equation}
M_\lambda(f^{in})\ge C_{\ref{cst4}} := \sqrt{\frac{M_\lambda(S)}{4^{1-\lambda} k_2}}\,, \label{inv5}
\end{equation}
then
\begin{equation*}
M_\lambda(f_\delta(t)) \ge C_{\ref{cst4}} > 0\,, \qquad t\ge 0\,.
\end{equation*}
\end{lemma}
%%%%%%%%%%%%%%%%

\begin{proof}
Let $t>0$. Owing to \eqref{wFd},
\begin{equation*}
\frac{\mathrm{d}}{\mathrm{d}t} M_\lambda(f_\delta(t)) - \frac{1}{2} \int_0^\infty \int_0^\infty \chi_{\lambda}(x,y) K(x,y) f_\delta(t,x) f_\delta(t,y)\ \mathrm{d}y \mathrm{d}x = M_\lambda(S_\delta) - 2 \delta M_\lambda(f_\delta(t))\,.
\end{equation*}
On the one hand, by \eqref{hypK}, \eqref{D4}, and \eqref{D5},
\begin{equation*}
- \chi_{\lambda}(x,y) K(x,y) \le k_2\left[ x^\lambda + y^\lambda - (x+y)^\lambda \right] \left( x^\lambda + y^\lambda \right) \le 2^{1-\lambda} k_2 (xy)^\lambda\,, \qquad (x,y)\in (0,\infty)^2\,.
\end{equation*}
On the other hand, it follows from \eqref{b2} and \eqref{n1} that there is $\delta_1\in (0,1)$ depending only on $S$ such that
\begin{equation*}
M_\lambda(S_\delta) \ge \frac{M_\lambda(S)}{2}\,, \qquad \delta\in (0,\delta_1)\,. 
\end{equation*}
Consequently, for $\delta\in (0,\delta_1)$, 
\begin{equation}
\frac{\mathrm{d}}{\mathrm{d}t} M_\lambda(f_\delta(t)) + F_\delta\left( M_\lambda(f_\delta(t)) \right) \ge \frac{M_\lambda(S)}{2}\,, \label{b22}
\end{equation}
with
\begin{equation*}
F_\delta(z) := 2^{-\lambda} k_2 z^2 + 2 \delta z\,, \qquad z\in \mathbb{R}\,.
\end{equation*}
Since $F_\delta$ is increasing and maps $[0,\infty)$ onto $[0,\infty)$, there is a unique $z_\delta>0$ such that $F_\delta(z_\delta) = M_\lambda(S)/2$, which is here explicitly given by 
\begin{equation*}
z_\delta := \frac{\sqrt{k_2 M_\lambda(S) + 2^{1+\lambda} \delta^2} - 2^{(1+\lambda)/2} \delta}{2^{(1-\lambda)/2} k_2}\,.
\end{equation*}
We then infer from \eqref{b22} and the comparison principle that
\begin{equation}
M_\lambda(f_\delta(t)) \ge \min\{ M_\lambda(f^{in}) , z_\delta\}\,, \qquad t\ge 0\,. \label{b23}
\end{equation}
Moreover, since 
\begin{equation*}
\lim_{\delta\to 0} z_\delta = \sqrt{\frac{M_\lambda(S)}{2^{1-\lambda}k_2}} >  2^{(\lambda-1)/2} \sqrt{\frac{M_\lambda(S)}{2^{1-\lambda} k_2}} = C_{\ref{cst4}}\,,
\end{equation*}
there is $\delta_0\in (0,\delta_1)$ such that $z_\delta \ge C_{\ref{cst4}}$ for $\delta\in (0,\delta_0)$. This property, together with \eqref{inv5} and \eqref{b23} completes the proof.
\end{proof}

%%%%%%%%%%%%%%%%
%%%%%%%%%%%%%%%%
\subsection{Uniform integrability}\label{sec3.2}
%%%%%%%%%%%%%%%%
%%%%%%%%%%%%%%%%

The next step is devoted to uniform integrability estimates. 
\refstepcounter{NumConst}\label{cst5}

%%%%%%%%%%%%%%%%
\begin{lemma}\label{lemb5}
There is $C_{\ref{cst5}}>0$ such that, if  $\delta\in (0,\delta_0)$ and $f^{in}$ satisfies \eqref{inv5} as well as
\begin{equation}
\int_0^\infty \Phi(f^{in}(x))\ \mathrm{d}x \le C_{\ref{cst5}}\,, \label{inv6}
\end{equation}
the function $\Phi$ being defined in \eqref{v0}, then
\begin{equation*}
\int_0^\infty \Phi(f_\delta(t,x))\ \mathrm{d}x \le C_{\ref{cst5}}\,, \qquad t\ge 0\,.
\end{equation*}
\end{lemma}
%%%%%%%%%%%%%%%%

\begin{proof}
Since $K(x-y,y) \le K(x,y)$ for $(x,y)\in (0,\infty)^2$ by \eqref{monK}, it follows from \cite[Lemma~8.2.18]{BLL2019} that
\begin{align*}
J_\delta(t) & := \int_0^\infty \Phi'(f_\delta(t,x)) \mathcal{C} f_\delta(t,x)\ \mathrm{d}x \\
& \le - \frac{1}{2} \int_0^\infty \int_x^\infty K(x,y) \left[ f_\delta \Phi'(f_\delta) - \Phi(f_\delta) \right](t,x) f_\delta(t,y)\ \mathrm{d}y \mathrm{d}x \\
& \quad - \frac{1}{2} \int_0^\infty \int_0^\infty K(x,y) f_\delta(t,x) \Phi'(f_\delta(t,x)) f_\delta(t,y)\ \mathrm{d}y \mathrm{d}x\,.
\end{align*}
By \eqref{hypK}, $K(x,y) \ge k_1 y^\lambda$, $(x,y)\in (0,\infty)^2$, and the properties of $\Phi$ guarantee that $r\Phi'(r) \ge \Phi(r)\ge 0$, $r\ge 0$, so that we further obtain
\begin{align*}
J_\delta(t) & \le - \frac{1}{2} \int_0^\infty \int_0^\infty K(x,y) \Phi(f_\delta(t,x)) f_\delta(t,y)\ \mathrm{d}y \mathrm{d}x \\
& \le - \frac{k_1}{2} \int_0^\infty \int_0^\infty y^\lambda \Phi(f_\delta(t,x)) f_\delta(t,y)\ \mathrm{d}y \mathrm{d}x \\
& = - \frac{k_1}{2} M_\lambda(f_\delta(t)) \int_0^\infty \Phi(f_\delta(t,x))\ \mathrm{d}x\,.
\end{align*}
Hence, owing to \eqref{inv1} and Lemma~\ref{lemb4b}, \refstepcounter{NumConst}\label{cst6}
\begin{equation*}
J_\delta(t) \le - 2 C_{\ref{cst6}} \int_0^\infty \Phi(f_\delta(t,x))\ \mathrm{d}x \;\;\text{ with }\;\; C_{\ref{cst6}} := \frac{\min\{2 , k_1 C_{\ref{cst4}}\}}{4} \in (0,1)\,.
\end{equation*}
We then infer from \eqref{n1}, \eqref{b4a}, and the non-negativity of $\Phi'$ that
\begin{align}
\frac{\mathrm{d}}{\mathrm{d}t}  \int_0^\infty \Phi(f_\delta(t,x))\ \mathrm{d}x & = \int_0^\infty \Phi'(f_\delta(t,x)) \partial_t f_\delta(t,x)\ \mathrm{d}x \nonumber \\
& = J_\delta(t) + \int_0^\infty \Phi'(f_\delta(t,x)) S_\delta(x)\ \mathrm{d}x \nonumber \\
& \le - 2 C_{\ref{cst6}} \int_0^\infty \Phi(f_\delta(t,x))\ \mathrm{d}x + \int_0^\infty \Phi'(f_\delta(t,x)) S(x)\ \mathrm{d}x\,. \label{b24}
\end{align}
Recalling that the properties of $\Phi$ implies that 
\begin{align*}
s \Phi'(r) & \le \Phi(r)+\Phi(s)\,, \qquad (r,s)\in [0,\infty)^2\,, \\
\Phi(sr) & \le \max\{1,s^2\} \Phi(r)\,, \qquad (r,s)\in [0,\infty)^2\,,
\end{align*}
see \cite[Proposition~7.1.9~(b) \&~(d)]{BLL2019}, we find
\begin{align*}
\int_0^\infty \Phi'(f_\delta(t,x)) S(x)\ \mathrm{d}x & = C_{\ref{cst6}} \int_0^\infty \Phi'(f_\delta(t,x)) \frac{S(x)}{C_{\ref{cst6}}}\ \mathrm{d}x \\
& \le C_{\ref{cst6}} \int_0^\infty \Phi(f_\delta(t,x))\ \mathrm{d}x + C_{\ref{cst6}} \int_0^\infty \Phi\left( \frac{S(x)}{C_{\ref{cst6}}} \right)\ \mathrm{d}x \\
& \le C_{\ref{cst6}} \int_0^\infty \Phi(f_\delta(t,x))\ \mathrm{d}x + \frac{1}{C_{\ref{cst6}}} \int_0^\infty \Phi(S(x))\ \mathrm{d}x \,.
\end{align*}
Combining the above inequality with \eqref{v2} and \eqref{b24} leads us to the differential inequality
\begin{equation*}
\frac{\mathrm{d}}{\mathrm{d}t} \int_0^\infty \Phi(f_\delta(t,x))\ \mathrm{d}x + C_{\ref{cst6}} \int_0^\infty \Phi(f_\delta(t,x))\ \mathrm{d}x \le \frac{L_\Phi(S)}{C_{\ref{cst6}}}\,,
\end{equation*}
from which we deduce that
\begin{align*}
\int_0^\infty \Phi(f_\delta(t,x))\ \mathrm{d}x & \le e^{-C_{\ref{cst6}} t} \int_0^\infty \Phi(f^{in}(x))\ \mathrm{d}x + \frac{L_\Phi(S)}{C_{\ref{cst6}}^2} \left( 1 - e^{-C_{\ref{cst6}} t} \right) \\
& \le \max\left\{ \int_0^\infty \Phi(f^{in}(x))\ \mathrm{d}x , \frac{L_\Phi(S)}{C_{\ref{cst6}}^2}\right\}\,.
\end{align*}
Lemma~\ref{lemb5} is now a straightforward consequence of \eqref{inv6} and the above inequality with $C_{\ref{cst5}}=L_\Phi(S)/C_{\ref{cst6}}^2$.
\end{proof}

%%%%%%%%%%%%%%%%
%%%%%%%%%%%%%%%%
\subsection{Proof of Proposition~\ref{prb0}}\label{sec3.3}
%%%%%%%%%%%%%%%%
%%%%%%%%%%%%%%%%

We fix $\delta\in (0,\delta_0)$ and consider the subset $\mathcal{Z}_\delta$ of $X_0 = L^1(0,\infty)$ defined by: $h\in \mathcal{Z}_\delta$ if and only if $h\in X_0^+$ satisfies
\begin{subequations}\label{c1}
\begin{align}
& C_{\ref{cst4}} \le M_\lambda(h) \le C_{\ref{cst1}} \,, \qquad \int_0^\infty \Phi(h(x))\ \mathrm{d}x \le C_{\ref{cst5}}\,, \label{c1a}\\
& M_m(h) \le \frac{M_m(S)}{2\delta}\,, \qquad m\in [0,1)\,, \label{c1b}\\
& M_1(h) \le \frac{M_\lambda(S)}{2\delta^{2-\lambda}}\,, \qquad M_{1+\lambda}(h) \le \frac{C_{\ref{cst3}}}{\delta^{(4+\lambda-\lambda^2)/(1-\lambda)}} \,. \label{c1c}
\end{align} 
\end{subequations} 
On the one hand, given $f^{in}\in \mathcal{Z}_\delta$ and $t\ge 0$, it follows from Lemma~\ref{lemb1}, Lemma~\ref{lemb4b}, and Lemma~\ref{lemb5} that $f_\delta(t)=\Psi_\delta(t,f^{in})$ satisfies \eqref{c1a} and from Lemma~\ref{lemb2} and  Lemma~\ref{lemb4} that it satisfies \eqref{c1b}. Furthermore, $\Psi_\delta(t,f^{in})$ satisfies \eqref{c1c} due to Lemma~\ref{lemn1}. Consequently, $\Psi_\delta(t,f^{in})\in \mathcal{Z}_\delta$ for all $t\ge 0$, so that $\mathcal{Z}_\delta$ is a positive invariant set for the semi-flow $\Psi_\delta$. On the other hand, $\mathcal{Z}_\delta$ is non-empty since $C_{\ref{cst4}} < C_{\ref{cst1}}$ by \eqref{inv1} and \eqref{inv5}. Moreover, owing to the superlinearity \eqref{v1} of $\Phi$, the Dunford-Pettis theorem ensures that $\mathcal{Z}_\delta$ is a closed convex and sequentially weakly compact subset of $L^1(0,\infty)$. Since $\Psi_\delta$ is a semi-flow on $\mathcal{Z}_\delta$ endowed with its weak topology by Proposition~\ref{prb00}, it follows from \cite[Theorem~1.2]{EMRR2005} that there is $\varphi_\delta\in \mathcal{Z}_\delta$ such that $\Psi_\delta(t,\varphi_\delta) = \varphi_\delta$ for all $t\ge 0$; that is, $\varphi_\delta\in \mathcal{Z}_\delta$ is a stationary solution to \eqref{b4a}. Since $\Psi_\delta(t,\varphi_\delta) = \varphi_\delta$ for all $t\ge 0$, we infer from \eqref{b6} that
\begin{equation*}
k_1 M_0(\varphi_\delta) M_\lambda(\varphi_\delta) \le \frac{M_0(\varphi_\delta)}{t} + M_0(S)
\end{equation*}
for all $t>0$. Hence, taking the limit $t\to\infty$,
\begin{equation*}
k_1 M_0(\varphi_\delta) M_\lambda(\varphi_\delta) \le M_0(S)\,,
\end{equation*}
from which we deduce, thanks to the lower bound for $M_\lambda(\varphi_\delta)$ in \eqref{c1},
\begin{equation}
k_1  C_{\ref{cst4}} M_0(\varphi_\delta) \le M_0(S)\,. \label{c2}
\end{equation}
Similarly, for $\mu\in (0,(1+\lambda)/2)$, it follows from \eqref{b11} with $m=(2\mu+1-\lambda)/2\in (0,1)$ that
\begin{equation*}
\left( \int_1^\infty x^\mu \varphi_\delta(x)\ \mathrm{d}x \right)^2 \le C_{\ref{cst2}}((2\mu+1-\lambda)/2,\mu) \left( \frac{M_{(2\mu+1-\lambda)/2}(\varphi_\delta)}{t} + M_{(2\mu+1-\lambda)/2}(S) \right)
\end{equation*}
for all $t>0$. Letting $t\to\infty$ gives
\begin{equation*}
\left( \int_1^\infty x^\mu \varphi_\delta(x)\ \mathrm{d}x  \right)^2 \le C_{\ref{cst2}}((2\mu+1-\lambda)/2,\mu)  M_{(2\mu+1-\lambda)/2}(S)\,. 
\end{equation*}
Together with \eqref{c2}, the above estimate entails that
 \refstepcounter{NumConst}\label{cst7}
\begin{align}
M_\mu(\varphi_\delta) & = \int_0^1 x^\mu \varphi_\delta(x)\ \mathrm{d}x + \int_1^\infty x^\mu \varphi_\delta(x)\ \mathrm{d}x \nonumber \\ 
& \le M_0(\varphi_\delta) + \sqrt{C_{\ref{cst2}}((2\mu+1-\lambda)/2,\mu)  M_{(2\mu+1-\lambda)/2}(S)} \nonumber \\
& \le C_{\ref{cst7}}(\mu) := \frac{M_0(S)}{k_1 C_{\ref{cst4}}} + \sqrt{C_{\ref{cst2}}((2\mu+1-\lambda)/2,\mu)  M_{(2\mu+1-\lambda)/2}(S)}\,. \label{c3}
\end{align}
Collecting the estimates \eqref{c1a}, \eqref{c2}, and \eqref{c3} gives \eqref{c4} and \eqref{c5} and completes the proof of Proposition~\ref{prb0}.

%%%%%%%%%%%%%%%%
%%%%%%%%%%%%%%%%
\section{Existence}\label{sec4}
%%%%%%%%%%%%%%%%
%%%%%%%%%%%%%%%%

\begin{proof}[Proof of Theorem~\ref{thmD0}~(a)]
Since $\Phi$ is superlinear at infinity by \eqref{v1}, it follows from \eqref{c4}, \eqref{c5}, and the Dunford-Pettis theorem that $(\varphi_\delta)_{\delta\in (0,\delta_0)}$ is relatively sequentially weakly compact in $X_0$. In turn, this compactness property and \eqref{c5} imply that $(\varphi_\delta)_{\delta\in (0,\delta_0)}$ is actually relatively sequentially weakly compact in $X_\mu$ for any $\mu\in [0,(1+\lambda)/2)$. Consequently, using a diagonal process, there are a subsequence $(\varphi_{\delta_j})_{j\ge 2}$ of $(\varphi_\delta)_{\delta\in (0,\delta_0)}$ and 
\begin{equation}
\varphi\in \bigcap_{\mu\in [0,(1+\lambda)/2)} X_{\mu}^+ \label{e1}
\end{equation}
such that, as $j\to\infty$,
\begin{equation}
\varphi_{\delta_j} \rightharpoonup \varphi \;\;\text{ in }\;\; X_\mu\,, \qquad \mu\in \left[ 0 , \frac{1+\lambda}{2} \right)\,. \label{e2}
\end{equation}
Since $\lambda\in [0,(1+\lambda)/2)$, it readily follows from \eqref{e2} that $\left( [(x,y)\mapsto \varphi_{\delta_j}(x)\varphi_{\delta_j}(y)] \right)_{j\ge 2}$ converges weakly to $[(x,y)\mapsto \varphi(x)\varphi(y)]$ in $X_{0,\lambda}\times X_{0,\lambda}$ as $j\to\infty$. It is then straightforward to pass to the limit $j\to\infty$ in the identity \eqref{n3} satisfied by $\varphi_{\delta_j}$ and deduce that $\varphi$ satisfies \eqref{D1}, thereby completing the proof of Theorem~\ref{thmD0}~(a), recalling that the other integrability properties of $\varphi$ listed there follow from Propositions~\ref{prc1} and~\ref{prc2}.
\end{proof}

%%%%%%%%%%%%%%%%
%%%%%%%%%%%%%%%%
\section{Non-existence}\label{sec5}
%%%%%%%%%%%%%%%%
%%%%%%%%%%%%%%%%

\begin{proof}[Proof of Theorem~\ref{thmD0}~(b)]
The proof relies on the same device as those of Propositions~\ref{prc1} and~\ref{prD2}. For $A>0$ and $x>0$, we set $\vartheta_A(x) = \min\{x,A\}$. We infer from \eqref{D1} and the symmetry of $K$ that
\begin{equation}
\begin{split}
\int_0^\infty \vartheta_A(x) S(x)\ \mathrm{d}x & = \frac{1}{2} \int_0^A \int_{A-x}^A (x + y - A) K(x,y) \varphi(x) \varphi(y)\ \mathrm{d}y\mathrm{d}x \\ 
& \quad + \int_0^A \int_A^\infty x K(x,y) \varphi(x) \varphi(y)\ \mathrm{d}y\mathrm{d}x \\
& \quad + \frac{A}{2} \int_A^\infty \int_A^\infty  K(x,y) \varphi(x) \varphi(y)\ \mathrm{d}y\mathrm{d}x \,.
\end{split} \label{D7}
\end{equation} 
	
We are left with identifying the limit as $A\to\infty$ of each term on the right hand side of \eqref{D7}. We first infer from \eqref{hypK} that
\begin{align*}
0 & \le \mathbf{1}_{(0,A)}(x) \mathbf{1}_{(A-x,A)}(y) (x+y-A) K(x,y) \varphi(x) \varphi(y) \\
& \le k_2  (x+y-A) \left( x^\lambda+y^\lambda \right) \varphi(x) \varphi(y) \\
& \le k_2 \left( x^\lambda y + y^\lambda x \right) \varphi(x)\varphi(y)\,.
\end{align*}
Since $\varphi\in X_1\cap X_\lambda\subset X_0\cap X_\lambda$ due to $\lambda\ge 1$ and 
\begin{equation*}
\lim_{A\to\infty} \mathbf{1}_{(0,A)}(x) \mathbf{1}_{(A-x,A)}(y) = 0\,, \qquad (x,y)\in (0,\infty)^2\,,
\end{equation*}
it follows from Lebesgue's dominated convergence theorem that
\begin{equation*}
\lim_{A\to\infty} \frac{1}{2} \int_0^A \int_{A-x}^A (x+y-A) K(x,y) \varphi(x) \varphi(y)\ \mathrm{d}y\mathrm{d}x = 0\,.
\end{equation*}
Next, using once more \eqref{hypK},
\begin{align*}
0 & \le \int_0^A \int_A^\infty x K(x,y) \varphi(x) \varphi(y)\ \mathrm{d}y\mathrm{d}x \le k_2 \int_0^A \int_A^\infty x \left( x^\lambda + y^\lambda \right) \varphi(x) \varphi(y)\ \mathrm{d}y\mathrm{d}x \\
& \le 2 k_2 \int_0^A \int_A^\infty xy^\lambda \varphi(x) \varphi(y)\ \mathrm{d}y\mathrm{d}x \le 2 k_2 M_1(\varphi) \int_A^\infty y^\lambda \varphi(y)\ \mathrm{d}y
\end{align*}
and
\begin{align*}
0 & \le A \int_A^\infty \int_A^\infty K(x,y) \varphi(x) \varphi(y)\ \mathrm{d}y\mathrm{d}x \le A k_2 \int_A^\infty \int_A^\infty \left( x^\lambda + y^\lambda \right) \varphi(x) \varphi(y)\ \mathrm{d}y\mathrm{d}x \\
& \le k_2 \int_A^\infty \int_A^\infty \left( x^\lambda y + x y^\lambda \right) \varphi(x) \varphi(y)\ \mathrm{d}y\mathrm{d}x \le 2 k_2 M_1(\varphi) \int_A^\infty y^\lambda \varphi(y)\ \mathrm{d}y\,,
\end{align*}
from which we deduce that
\begin{align*}
& \lim_{A\to\infty} \int_0^A \int_A^\infty x K(x,y) \varphi(x) \varphi(y)\ \mathrm{d}y\mathrm{d}x = 0\,, \\
& \lim_{A\to\infty} \frac{A}{2} \int_A^\infty \int_A^\infty  K(x,y) \varphi(x) \varphi(y)\ \mathrm{d}y\mathrm{d}x = 0\,,
\end{align*}
recalling that $\varphi\in X_\lambda$. Collecting the above information, we may take the limit $A\to\infty$ in \eqref{D7} and conclude that
\begin{equation*}
\lim_{A\to \infty} \int_0^\infty \min\{x,A\} S(x)\ \mathrm{d}x = 0\,.
\end{equation*}
Hence, $S\equiv 0$ which, together with \eqref{D2a}, implies that $\varphi\equiv 0$ as well.
\end{proof}

%%%%%%%%%%%%%%%%
%%%%%%%%%%%%%%%%
\section*{Acknowledgments}
%%%%%%%%%%%%%%%%
%%%%%%%%%%%%%%%%

Part of this work was done while enjoying the support and hospitality of the Hausdorff Research Institute for Mathematics within the Junior Trimester Program \textsl{Kinetic Theory}. I also thank Marina~A. Ferreira and Juan~J.L. Vel\'azquez for motivating discussions on the topic studied in this paper.

%%%%%%%%%%%%%%%%
%%%%%%%%%%%%%%%%
\bibliographystyle{siam}
\bibliography{StatSolCoagSource}

\begin{thebibliography}{10}

\bibitem{BaCa1990}
{\sc J.~M. Ball and J.~Carr}, {\em The discrete coagulation-fragmentation
  equations: {E}xistence, uniqueness, and density conservation}, J. Statist.
  Phys., 61 (1990), pp.~203--234.

\bibitem{BLL2019}
{\sc J.~Banasiak, W.~Lamb, and {\relax Ph}.~Lauren{\c c}ot}, {\em Analytic
  methods for coagulation-fragmentation models}, {CRC Press}, 2019.

\bibitem{Buro1983}
{\sc A.~V. Burobin}, {\em Existence and uniqueness of the solution of the
  {C}auchy problem for a spatially nonhomogeneous coagulation equation},
  Differ. Uravn., 19 (1983), pp.~1568--1579.

\bibitem{dlVP15}
{\sc C.~De~La Vall{\'e}e~Poussin}, {\em Sur l'int\'egrale de {L}ebesgue},
  Trans. Amer. Math. Soc., 16 (1915), pp.~435--501.

\bibitem{Dubo1994b}
{\sc P.~B. Dubovskii}, {\em Mathematical theory of coagulation}, vol.~23 of
  Lecture Notes Series, Seoul National University, Research Institute of
  Mathematics, Global Analysis Research Center, Seoul, 1994.

\bibitem{EsMi2006}
{\sc M.~Escobedo and S.~Mischler}, {\em Dust and self-similarity for the
  {S}moluchowski coagulation equation}, Ann. Inst. H. Poincar\'e Anal. Non
  Lin\'eaire, 23 (2006), pp.~331--362.

\bibitem{EMRR2005}
{\sc M.~Escobedo, S.~Mischler, and M.~Rodriguez~Ricard}, {\em On
  self-similarity and stationary problem for fragmentation and coagulation
  models}, Ann. Inst. H. Poincar\'e Anal. Non Lin\'eaire, 22 (2005),
  pp.~99--125.

\bibitem{FLNV}
{\sc M.~A. Ferreira, J.~Lukkarinen, A.~Nota, and J.~J.~L. Vel{\' a}zquez}, {\em
  Stationary non-equilibrium solutions for coagulation systems}.
\newblock arXiv:1909.10608, 2019.

\bibitem{FoLa2005}
{\sc N.~Fournier and {\relax Ph}.~Lauren{\c{c}}ot}, {\em Existence of
  self-similar solutions to {S}moluchowski's coagulation equation}, Comm. Math.
  Phys., 256 (2005), pp.~589--609.

\bibitem{KuTh2019}
{\sc C.~Kuehn and S.~Throm}, {\em Smoluchowski's discrete coagulation equation
  with forcing}, NoDEA Nonlinear Differential Equations Appl., 26 (2019).
\newblock Paper No. 17, 33p.

\bibitem{Laur2019}
{\sc {\relax Ph}.~Lauren\c{c}ot}, {\em Stationary solutions to
  coagulation-fragmentation equations}, Ann. Inst. H. Poincar\'{e} Anal. Non
  Lin\'{e}aire, 36 (2019), pp.~1903--1939.

\bibitem{LaMi02c}
{\sc {\relax Ph}.~Lauren{\c{c}}ot and S.~Mischler}, {\em The continuous
  coagulation-fragmentation equations with diffusion}, Arch. Ration. Mech.
  Anal., 162 (2002), pp.~45--99.

\bibitem{Le1977}
{\sc C.~H. L{\^e}}, {\em Etude de la classe des op\'erateurs {$m$}-accr\'etifs
  de {$L^{1}(\Omega )$} et accr\'etifs dans {$L^{\infty }(\Omega )$}}, PhD
  thesis, Universit{\'e} de Paris {VI}, 1977.
\newblock Th{\`e}se de {$3^{\text{\`eme}}$} cycle.

\bibitem{LeTs1981a}
{\sc F.~Leyvraz and H.~R. Tschudi}, {\em Singularities in the kinetics of
  coagulation processes}, J. Phys. A, 14 (1981), pp.~3389--3405.

\bibitem{McLe1962c}
{\sc J.~B. McLeod}, {\em On an infinite set of non-linear differential
  equations}, Quart. J. Math. Oxford Ser. (2), 13 (1962), pp.~119--128.

\bibitem{McLe1964}
\leavevmode\vrule height 2pt depth -1.6pt width 23pt, {\em On the scalar
  transport equation}, Proc. London Math. Soc. (3), 14 (1964), pp.~445--458.

\bibitem{Melz1957b}
{\sc Z.~A. Melzak}, {\em A scalar transport equation}, Trans. Amer. Math. Soc.,
  85 (1957), pp.~547--560.

\bibitem{MiRR03}
{\sc S.~Mischler and M.~Rodriguez~Ricard}, {\em Existence globale pour
  l'\'equation de {S}moluchowski continue non homog\`ene et comportement
  asymptotique des solutions}, C. R. Math. Acad. Sci. Paris, 336 (2003),
  pp.~407--412.

\bibitem{NiVe2013a}
{\sc B.~Niethammer and J.~J.~L. Vel{\'a}zquez}, {\em Self-similar solutions
  with fat tails for {S}moluchowski's coagulation equation with locally bounded
  kernels}, Comm. Math. Phys., 318 (2013), pp.~505--532.

\bibitem{SvR2002}
{\sc M.~Shirvani and H.~J. Van~Roessel}, {\em Existence and uniqueness of
  solutions of {S}moluchowski's coagulation equation with source terms}, Quart.
  Appl. Math., 60 (2002), pp.~183--194.

\bibitem{Simo1998}
{\sc S.~Simons}, {\em On the solution of the coagulation equation with a
  time-dependent source-application to pulsed injection}, J. Phys. A, 31
  (1998), pp.~3759--3768.

\bibitem{Smol1916}
{\sc M.~v. {Smoluchowski}}, {\em Drei {V}ortr{\"a}ge \"uber {D}iffusion,
  {B}rownsche {B}ewegung und {K}oagulation von {K}olloidteilchen}, Physik.
  Zeitschr., 17 (1916), pp.~557--571, 585--599.

\bibitem{Smol1917}
\leavevmode\vrule height 2pt depth -1.6pt width 23pt, {\em Versuch einer
  mathematischen {T}heorie der {K}oagulationskinetik kolloider {L}{\"o}sungen},
  Zeitschrift f. phys. Chemie, 92 (1917), pp.~129--168.

\bibitem{Spou1984}
{\sc J.~L. Spouge}, {\em An existence theorem for the discrete
  coagulation-fragmentation equations}, Math. Proc. Cambridge Philos. Soc., 96
  (1984), pp.~351--357.

\bibitem{Spou1985b}
\leavevmode\vrule height 2pt depth -1.6pt width 23pt, {\em An existence theorem
  for the discrete coagulation-fragmentation equations. {II}. {I}nclusion of
  source and efflux terms}, Math. Proc. Cambridge Philos. Soc., 98 (1985),
  pp.~183--185.

\bibitem{Stew1989}
{\sc I.~W. Stewart}, {\em A global existence theorem for the general
  coagulation-fragmentation equation with unbounded kernels}, Math. Methods
  Appl. Sci., 11 (1989), pp.~627--648.

\bibitem{vDEr1985}
{\sc P.~G.~J. van Dongen and M.~H. Ernst}, {\em Comment on ``{L}arge-time
  behavior of the {S}moluchowski equations of coagulation''}, Phys. Rev. A, 32
  (1985), pp.~670--672.

\bibitem{Whit1980}
{\sc W.~H. White}, {\em A global existence theorem for {S}moluchowski's
  coagulation equations}, Proc. Amer. Math. Soc., 80 (1980), pp.~273--276.

\end{thebibliography}
%%%%%%%%%%%%%%%%
%%%%%%%%%%%%%%%%

\end{document}